\documentclass[11pt,fleqn]{article}
\usepackage{amssymb}
\usepackage{amsfonts}
\usepackage{amsmath}
\usepackage{amsthm}
\usepackage{graphicx}
\usepackage{epsfig}
\usepackage{psfrag}
\usepackage{color}
\usepackage{dsfont}
\makeatletter
\xdef\@endgadget#1{{\unskip\nobreak\hfil\penalty50\hskip1em\hbox{}\nobreak
    \hfil#1\parfillskip=0pt\finalhyphendemerits=0\par}}
\def\@qedsymbol{${}_\blacksquare$}
\def\qed{\@endgadget{\@qedsymbol}}
\newtheorem{lemma}{Lemma}[section]
\newtheorem{theorem}[lemma]{Theorem}
\newtheorem{corollary}[lemma]{Corollary}
\newtheorem{example}[lemma]{Example}

\newtheorem{proposition}[lemma]{Proposition}
\newtheorem{remark}[lemma]{Remark}
\newcommand{\mR}{\mathbb{R}}

\newcommand{\X}{\mathcal{X}}

\newcommand{\Y}{\mathcal{Y}}
\newcommand{\U}{\mathcal{U}}

\newcommand{\bq}{\begin{equation}}
\newcommand{\eq}{\end{equation}}
\newcommand{\bma}{\begin{bmatrix}}
\newcommand{\ema}{\end{bmatrix}}

\def\BibTeX{{\rm B\kern-.05em{\sc i\kern-.025em b}\kern-.08em
    T\kern-.1667em\lower.7ex\hbox{E}\kern-.125emX}}

\title{\LARGE \bf Maximum Power Transfer for Nonlinear State Space Systems}

\author{Arjan van der Schaft
\thanks{A.J. van der Schaft is with the Bernoulli Institute for Mathematics, Computer
Science and AI, Jan C. Willems Center for Systems and Control, University of Groningen, PO Box 407, 9700 AK, the
Netherlands,
        {\tt\small A.J.van.der.Schaft@rug.nl}}
}

\begin{document}

\maketitle
\thispagestyle{empty}
\pagestyle{empty}

\begin{abstract}
The classical Maximum Power Transfer theorem of linear electrical network theory is generalized to the setting of a nonlinear state space system connected to a source. This yields a state space version of the input-output operator results of \cite{wyatt}. Key tool in the analysis is the formulation of a Hamiltonian input-output system, which is closely related to Pontryagin's Maximum principle. The adjoint variational system incorporated in this system defines an optimal load. The structure of such an optimal load is investigated for classes of physical systems.
\end{abstract}
%
%\begin{IEEEkeywords}
%Power transfer, nonlinear state space systems, Hamiltonian input-output systems. 
%\end{IEEEkeywords}
%
%\section{Introduction}
\section{Introduction}
%\begin{figure}[t]
%\begin{center}
%%\includegraphics[width=0.6\columnwidth]{Thevenin}
%\vspace{-2cm}
%\caption{Th\'evenin equivalent network.}
%\label{Thevenin}
%%\vspace{-2cm}
%\end{center}
%\end{figure}
Consider a general nonlinear input-state-output system
\bq
\label{sigma}
\Sigma: \quad
\begin{array}{rcll}
\dot{x} & = & f(x,u), \qquad &x \in\X, \, u \in \U=\mR^m \\[2mm]
y & = & h(x,u), \qquad & y \in \Y=\mR^m.
\end{array}
\eq
Furthermore, consider any vector signal $y_S:[0,T] \to \mR^m$ for some $T>0$. Suppose we want to \emph{maximize} 
\bq
\label{max}
\int_0^T (y_S(t)- y(t))^\top u(t) dt, \quad x(0)=x_0,
\eq
over all input functions $u: [0,T] \to \mR^m$. 
In a physical system setting this corresponds to determining the \emph{optimal load} of a nonlinear \emph{Th\'evenin} or \emph{Norton} equivalent system, given by the impedance/admittance $\Sigma$ together with a source signal $y_S$. Here 'optimal' is understood as extracting the maximal energy on the time-interval $[0,T]$. In fact, in the Th\'evenin case, $\Sigma$ is an impedance given by an electrical network of the form
\bq
\begin{array}{rcll}
\dot{x} & = & f(x,I), \qquad &x \in \mR^n, \, I \in \mR^m, \\[2mm]
V & = & h(x,I), \qquad & V \in \mR^m,
\end{array}
\eq
with input vector $I$ (current), output vector $V$ (voltage), and internal state variables $x$ (charges and flux linkages), in \emph{series} interconnection with a voltage source signal $y_S=V_S:[0,T] \to \mR^m$. Alternatively, in the Norton case, one considers an impedance \eqref{sigma} with input $V$ and output $I$, in \emph{parallel} interconnection with a current source $y_S=I_S:[0,T] \to \mR^m$. In case the system is \emph{linear} and given by a transfer matrix, the solution to this problem is well-known, and is commonly referred to as the \emph{Maximum Power Transfer theorem}.

In the \emph{nonlinear} case this problem was addressed before in \cite{wyatt0, wyatt}, where however $\Sigma$ is represented by a nonlinear \emph{input-output operator} instead of a state space model as in \eqref{sigma}. In its turn, \cite{wyatt0,wyatt} continues the approach of \cite{desoer} in the linear case. Here, instead of finding the impedance of the optimal load directly, first the optimal current $\hat{I}$ (in the Th\'evenin case) is determined. Since the voltage across the load is equal to the difference of $y_S$ and the voltage delivered by $\Sigma$ with input $\hat{I}$, this indirectly defines the impedance of an optimal load. 

The main contribution of the present paper is to provide a \emph{state space} version of the results in \cite{wyatt0,wyatt}; thereby staying much closer to the physical models. Key tool is a Hamiltonian input-output system derived from $\Sigma$, which closely relates to Pontryagin's Minimum principle. It turns out that an optimal load is determined by the \emph{adjoint variational system} along an optimal trajectory. The structure of these adjoint variational systems, and therefore of optimal loads, will be analyzed for two physical system classes $\Sigma$ in Section \ref{special}.

\section{Main results}
Obviously, maximizing \eqref{max} is the same as \emph{minimizing} over $u: [0,T] \to \mR^m$ the expression
\bq
\label{min}
P_{x_0,y_S}(u) :=\int_0^T (y(t)- y_S(t))^\top u(t) dt, \quad x(0)=x_0,
\eq
where $y: [0,T] \to \mR^p$ is the output of \eqref{sigma} corresponding to fixed initial condition $x_0$. For ease of notation we will often abbreviate the notation $P_{x_0,y_S}(u)$ to $P(u)$. Minimization of \eqref{min} corresponds to an \emph{optimal control} problem, which is \emph{time-varying} in case $y_S$ is a non-trivial function of time $t$. In fact, Pontryagin's Minimum principle tells to consider the Hamiltonian
\bq
H(x,p,u,t) = p^\top f(x,u) + \left(h(x,u) -y_S(t)\right)^\top u,
\eq
and to write down the corresponding Hamiltonian equations, with mixed boundary conditions $x(0)=x_0$ and $p(T)=0$ (or first to add an extra state variable $t$). Instead, we will follow a \emph{different}, more direct, route, which is leading to closely related equations from a different point of view. This alternative route allows one to deal with arbitrary signals $y_S$, and to give explicit conditions for attaining a \emph{global minimum} (different from the ones provided by Pontryagin's Minimum principle). Furthermore, it provides the optimal input as the output of a system with input $y_S$.

The approach starts by defining along the solutions $(u,x,y)$ of \eqref{sigma} the \emph{variational systems}\footnote{The Jacobian matrix $\frac{\partial f}{\partial x^\top}(x,u)$ denotes the matrix with $(i,j)$-th element given by $\frac{\partial f_i}{\partial x_j}(x,u)$. Furthermore $\frac{\partial f^\top}{\partial x}(x,u)$ is the transposed matrix with $(i,j)$-th element equal to $\frac{\partial f_j}{\partial x_i}(x,u)$. Similarly for the other Jacobian matrices.} \cite{crouch}
\bq
\delta \Sigma: \quad \begin{array}{rcl}
\dot{\delta x} & = & \frac{\partial f}{\partial x^\top}(x,u) \delta x + \frac{\partial f}{\partial u^\top}(x,u)\delta u\\[2mm]
\delta y & = & \frac{\partial h}{\partial x^\top}(x,u) \delta x+ \frac{\partial h}{\partial u^\top}(x,u) \delta u
\end{array}
\eq
Note \cite{crouch} that the system $\Sigma$ together with its variational systems $\delta \Sigma$ admits a coordinate-free and global description with $(x,\delta x) \in T\X$, $(u,\delta u) \in T\U$, $(y,\delta y) \in T\Y$.

Furthermore define the \emph{adjoint variational systems}
\bq
\label{adjoint}
\delta^* \Sigma: \quad \begin{array}{rcl}
\dot{p} & = & -\frac{\partial f^\top}{\partial x}(x,u) p - \frac{\partial h^\top}{\partial x}(x,u) u_a\\[2mm]
y_a & = & \frac{\partial f^\top}{\partial u}(x,u) p + \frac{\partial h^\top}{\partial u}(x,u) u_a,
\end{array}
\eq
which are uniquely characterized \cite{crouch} by the property 
\bq
\label{unique}
\frac{d}{dt} p^\top (t) \delta x (t) = y_a^\top (t) \delta u(t) - u_a^\top (t) \delta y(t)
\eq
along all solutions $(u,x,y)$.
The system $\Sigma$ together with its adjoint variational systems $\delta^* \Sigma$ admits a coordinate-free and global description with $(x,p) \in T^*\X$, $(u,u_a) \in T^*\U$, $(y,y_a) \in T^*\Y$, cf. \cite{crouch}. 

Now consider the system $\Sigma$ \emph{together} with its adjoint variational systems $\delta^* \Sigma$, where additionally we let $u_a=u$ and define the output $y^+ := y + y_a$ ('parallel interconnection'). This amounts to the system
\bq
\label{+}
\Sigma^+: \quad 
\begin{array}{rcl} 
\dot{x} & = & \frac{\partial H^+}{\partial p}(x,p,u) \\[2mm]
\dot{p} & = & -\frac{\partial H^+}{\partial x}(x,p,u) \quad (x,p) \in T^*\X, \\[2mm]
y^+ & = & \frac{\partial H^+}{\partial u}(x,p,u),
\end{array}
\eq
where
\bq
H^+(x,p,u):= p^\top f(x,u) + u^\top h(x,u).
\eq
The system \eqref{+} is a \emph{Hamiltonian input-output (IO) system} in the sense of \cite{brockett,vds81}.
Note that the equations \eqref{+} for $y^+=0$ equal the \emph{first-order optimality conditions} provided by Pontryagin's Minimum principle for the cost criterium $u^\top h(x,u)$. Furthermore, because of the linearity of $H^+$ in $p$, it follows that \eqref{+} with mixed boundary conditions $x(0)=x_0, p(T)=0,$ has unique solutions  for every input function $u$.
If the partial Hessian matrix
\bq
\label{partialHessian}
\frac{\partial^2 H^+}{\partial u\partial u^\top}(x,p,u) 
\eq
has \emph{full rank}, then \eqref{+} admits an \emph{inverse}, which is given by another Hamiltonian IO system
\bq
\label{x}
\Sigma^\times: \quad 
\begin{array}{rcl} 
\dot{x} & = & \frac{\partial H^\times}{\partial p}(x,p,y^+) \\[2mm]
\dot{p} & = & -\frac{\partial H^\times}{\partial x}(x,p,y^+) \\[2mm]
u & = & - \frac{\partial H^\times}{\partial y^+}(x,p,y^+),
\end{array}
\eq
with inputs $y^+$ and outputs $u$. Its Hamiltonian $H^\times$ equals minus the \emph{partial Legendre transform} of $H^+(x,p,u)$ with respect to $u$, that is
\bq
\label{times}
\begin{array}{rcl}
H^\times(x,p,y^+)&:= &H^+(x,p,u) - u^\top y^+ \\[2mm]
&= &p^\top f(x,u) + u^\top \left( h(x,u) - y^+ \right),
\end{array}
\eq
where $u$ is solved as a function of $x,p,y^+$ from the equations $y^+  =  \frac{\partial H^+}{\partial u}(x,p,u)$. 
%This motivates the following standing assumption guaranteeing the existence of an inverse system \eqref{x}.
%\begin{assumption}
%\label{ass}
%The partial Hessian matrix \eqref{partialHessian} has full rank for all $(x,p,u)$.
%\end{assumption}

Note that in case of a \emph{linear system}
\bq
\label{sigmalin}
\begin{array}{rcll}
\dot{x} & = & Ax +Bu, \qquad &x \in\mR^n, u \in \mR^m \\[2mm]
y & = & Cx + Du, \qquad & y \in \mR^p,
\end{array}
\eq
the variational systems are \emph{equal} to the system itself (but with state $\delta x$, input $\delta u$, and output $\delta y$), while the adjoint variational systems are the adjoint system
\bq
\label{adjointlin}
\begin{array}{rcll}
\dot{p} & = & -A^\top p  - C^\top u_a, \qquad &p \in\mR^n, u_a \in \mR^p \\[2mm]
y_a & = & B^\top p + D^\top u_a, \qquad & y_a \in \mR^m.
\end{array}
\eq
Furthermore, in the linear case the system $\Sigma^+$ amounts to
\bq
\begin{array}{rcl} 
\label{+lin}
\dot{x} & = & \frac{\partial H^+}{\partial p}(x,p,u) = Ax + Bu \\[2mm]
\dot{p} & = & -\frac{\partial H^+}{\partial x}(x,p,u) = - A^\top p - C^\top u \\[2mm]
y^+ & = & \frac{\partial H^+}{\partial u}(x,p,u) = Cx + Du + B^\top p + D^\top u,
\end{array}
\eq
with $H^+(x,p,u)=p^\top (Ax+ Bu) + u^\top (Cx + Du)$. 
%Denoting the transfer matrix of the linear system \eqref{sigmalin} by $G(s)=C(Is -A)^{-1}B +D$, the transfer matrix of \eqref{+lin} is given by $G(s) + G^\top(-s)$.
System \eqref{+lin} admits an inverse if and only if the matrix $D+D^\top$ is invertible, in which case 
\bq
H^\times(x,p,y^+) = p^\top (Ax + Bu) + u^\top \left(Cx + Du - y^+ \right)
\eq
with $u=(D+D^\top)^{-1} \left(y^+ - Cx - B^\top p\right)$.

Now, consider as before for a fixed time-function $y_S: [0,T] \to \mR^m$ and initial condition $x_0$ the functional $P_{x_0,y_S}(u)$ defined in \eqref{min}.
%=\int_0^T (y(t)- y_S(t))^\top u(t) dt, \quad x(0)=x_0,
%\eq
%with $u: [0,T] \to \mR^m$, where $y: [0,T] \to \mR^m$ is the output function corresponding to $u$ and initial condition $x_0$.
Furthermore, consider an arbitrary \emph{variation} $\delta u: [0,T] \to \mR^m$ of $u$, together with the function
\bq
\bar{P}(\lambda):= P_{x_0,y_S}(u + \lambda \delta u), \quad \lambda \in [0,1].
\eq
Then compute
\bq
\label{dlambda}
\begin{array}{rcl} 
\frac{d}{d \lambda} \bar{P}(\lambda) &= &\int_0^T (y(t)- y_S(t))^\top \delta u(t) dt \\[2mm]
&&+ \int_0^T \delta y(t)^\top \left(u(t) + \lambda \delta u(t)\right) dt,
\end{array}
\eq
where $y(t)$ is the output of \eqref{sigma} resulting from $u + \lambda \delta u$ and $x(0)=x_0$, and $\delta y(t)$ is the variational output resulting from the variational input $\delta u$ along the trajectories corresponding to $u + \lambda \delta u$ and $x(0)=x_0$. Now recall the equality \eqref{unique}.
Substitution of $u_a(t)= u(t) + \lambda \delta u (t)$ and the mixed boundary conditions $\delta x(0)=0, p(T)=0$ into \eqref{unique}, yields
\bq
\int_0^T \delta y(t)^\top \left(u(t) + \lambda \delta u(t)\right) dt = \int_0^T y_a^\top (t) \delta u(t) dt
\eq
Together with \eqref{dlambda} this yields
\bq
\label{lambda}
\begin{array}{rcl} 
\frac{d}{d \lambda} \bar{P}(\lambda) &= &\int_0^T (y(t) + y_a(t)- y_S(t))^\top \delta u(t) dt \\[2mm]
&= &\int_0^T (y^+(t)- y_S(t))^\top \delta u(t) dt,
\end{array}
\eq
where $y^+(t)$ is the output of \eqref{+} for input $u + \lambda \delta u$ and $x(0)=x_0, p(T)=0$.
In particular
\bq
\label{lambda0}
\frac{d}{d \lambda} \bar{P}(0) = \int_0^T (y^+(t)- y_S(t))^\top \delta u(t) dt
\eq
where $y^+(t)$ is the output of \eqref{+} for input $u$ and $x(0)=x_0,p(T)=0$.
\begin{remark}
The same computation implies that the \emph{variational derivative} of $P_{x_0,y_S}(u)$ is the function
\bq
y^+(t)- y_S(t), \quad t \in [0,T],
\eq
where $y^+(t)$ is the output of \eqref{+} for input $u(t), t \in [0,T]$, and $x(0)=x_0, p(T)=0$.
\end{remark}

Now suppose the Hessian matrix \eqref{partialHessian} has full rank, guaranteeing the existence of the inverse system \eqref{x} to \eqref{+}.
Then associated to $y_S:[0,T] \to \mR^m$ there exists a unique corresponding input function $\hat{u}: [0,T] \to \mR^m$, given as the output of \eqref{x} for input $y^+=y_S$ for $x(0)=x_0, p(T)=0$. This yields a corresponding state trajectory $\hat{x}: [0,T] \to \X$ with $\hat{x}(0)=x_0$, and corresponding output $\hat{y}: [0,T] \to \mR^m$, of \eqref{sigma}. Since $\hat{u}$ ensures that $y^+(t)=y_S(t), t \in [0,T]$, the expression \eqref{lambda0} for $u=\hat{u}$ reduces to $\frac{d}{d \lambda} \bar{P}(0)=0$ for any variation $\delta u$. Hence the input function $\hat{u}$ is \emph{extremizing} $P_{x_0,y_S}(u)$. This is summarized as follows.
\begin{proposition}
Consider the system \eqref{sigma} and $P_{x_0,y_S}(u)$ given by \eqref{min}. Assume the Hessian matrix \eqref{partialHessian} has full rank, and consider the inverse system \eqref{x} of \eqref{+}. Then the output $\hat{u}$ of \eqref{x} for input $y_S$ with $x(0)=x_0, p(T)=0,$ is extremizing $P_{x_0,y_S}(u)$.
\end{proposition}

In order to analyze when $\hat{u}$ \emph{globally minimizes} $P_{x_0,y_S}(u)$ we follow an argumentation similar to the one in \cite{wyatt}. First, let us recall the notion of \emph{incremental cylo-passivity}. A system \eqref{sigma} is \emph{incrementally passive}, cf. \cite{desoervidy, pass}, if given two copies $\Sigma^1$, $\Sigma^2$ of $\Sigma$ with inputs $u^1,u^2$, states $z^1,z^2$, and outputs $y^1,y^2$, there exists a nonnegative function $S(z^1,z^2)$ such that along the dynamics of the product system $\Sigma^1 \times \Sigma^2$
\bq
\frac{d}{dt}S(z^1,z^2) \leq (y^1 -y^2)^\top (u^1 - u^2).
\eq
This implies
\bq
\label{inq}
%\begin{array}{l}
\int_0^T \left[y^1(t) -y^2(t) \right]^\top \left[u^1(t) - u^2(t) \right] dt  \geq S(z^1(T),z^2(T))- S(z^1(0),z^2(0)).
%\end{array}
\eq
Since $\bar{P}(0) = P(\hat{u}), \bar{P}(1) = P(\hat{u} + \delta u)$, it follows that
\bq
P(\hat{u} + \delta u) - P(\hat{u}) = \int_0^1 \frac{d}{d \lambda} \bar{P}(\lambda) d \lambda.
\eq
Using again \eqref{lambda}, and writing $\lambda \delta u(t) = (\hat{u}(t) + \lambda \delta u(t)) - \hat{u}(t)$ this yields
%Then compute
%\bq
%\frac{d}{d \lambda} \bar{P}(\lambda) = \int_0^T (y(t)- y_S(t))^\top \delta u(t) dt + \int_0^T \delta y(t)^\top \left(\hat{u}(t) + \lambda \delta u(t)\right) dt
%\eq
%where $\delta y(t)$ is the variational output resulting from the variational input $ \delta u$ along the trajectories corresponding to $\hat{u} + \lambda \delta u$ and (of course) $x(0)=x_0$. Now recall the equality
%\bq
%\frac{d}{dt} p^\top (t) \delta x (t) = y_a^\top (t) \delta u(t) - u_a^\top (t) \delta y(t),
%\eq
%where we substitute $u_a(t)= \hat{u}(t) + \lambda \delta u (t)$, $\delta x(0)=0$, and $p(T)=0$, and thus
%\bq
%\int_0^T \delta y(t)^\top \left(\hat{u}(t) + \lambda \delta u(t)\right) dt = \int_0^T y_a^\top (t) \delta u(t) dt
%\eq
%with $y_a$ the corresponding solution of the adjoint variational system. Taken together this yields
%\bq
%\frac{d}{d \lambda} \bar{P}(\lambda) = \int_0^T (y(t) + y_a(t)- y_S(t))^\top \delta u(t) dt = \int_0^T (y^+(t)- y_S(t))^\top \delta u(t) dt
%\eq
%Now suppose $\hat{u}$ is 
\bq
\begin{array}{l}
P(\hat{u} + \delta u) - P(\hat{u})  = \int_0^1 \frac{1}{\lambda} \int_0^T (y^+(t)- y_S(t))^\top \lambda \delta u(t) dt \\[2mm]=
 \qquad \int_0^1 \frac{1}{\lambda} \int_0^T \left[y^+(t)- y_S(t)\right]^\top \left[(\hat{u}(t) + \lambda \delta u(t)) - \hat{u}(t) \right] dt.
\end{array}
\eq
This leads to the following theorem.
\begin{theorem}
Consider $\Sigma$ given by \eqref{sigma} with initial state $x(0)=x_0$, together with the system $\Sigma^+$ given by \eqref{+}, with mixed boundary conditions $x(0)=x_0$ and $p(T)=0$. Assume the Hessian matrix \eqref{partialHessian} has full rank and $\Sigma^+$ is incrementally passive with storage function $S(x^1,p^1,x^2,p^2)$ being zero at $(x_0,p_0,x_0,p_0)$ where $p_0=p(0)$. Consider an arbitrary source signal $y_S:[0,T] \to \mR^m$. Let $\hat{u}:[0,T] \to \mR^m$ be the output of the inverse system \eqref{x} for input $y^+=y_S$ and $x(0)=x_0$, $p(T)=0$. Then $\hat{u}$ is globally minimizing $P_{x_0,y_S}(u)$.
%for which there exists an input function $\hat{u}$ for $\Sigma^+$ reproducing $y^+(t)=y_S(t), t \in [0,T]$. Then $\hat{u}$ is minimizing $P(u)$.
\end{theorem}
\begin{proof} By incremental passivity $\int_0^T \left[y^+(t)- y_S(t)\right]^\top \! \! \left[(\hat{u}(t) + \lambda \delta u(t)) - \hat{u}(t) \right] \! dt\\ \geq 0$, see \eqref{inq}. Hence $P(\hat{u} + \delta u) - P(\hat{u}) \geq 0$ for any $\delta u$.
\end{proof}
%${}_\blacksquare$
Since a linear system is incrementally passive if and only if it is passive, this yields the following corollary.
 \begin{corollary}
If $\Sigma$ is a \emph{linear} system \eqref{sigmalin} then $\hat{u}$ is minimizing $P_{x_0,y_S}(u)$ if $D+D^\top$ is invertible and $\Sigma^+$ is passive.
\end{corollary}
%Note that minimizing $P(u) =\int_0^T (y(t)- y_S(t))^\top u(t) dt$ is the same as \emph{maximizing} $-P(u) =\int_0^T (y_S(t)- y(t))^\top u(t) dt$ over all $u$.

Thus if the IO Hamiltonian system \eqref{+} has an \emph{inverse}, which is \emph{incrementally passive}, then for any output signal $y_S$ there is a \emph{unique} $\hat{u}$ that is globally minimizing $P_{x_0,y_S}(u)$. This unique $\hat{u}$ is solving the optimal load problem for the Th\'evenin equivalent system defined by $\Sigma$ and $y_S$ in the following sense. Note that if $u=\hat{u}$ is the input to \eqref{+} then its output is $y^+(t)=y_S(t), t \in [0,T]$. This means that
\bq
\begin{array}{rcl}
y_S(t) &=&\frac{\partial H^+}{\partial u}(x,p,u) = \frac{\partial}{\partial u} \left(p^\top f(x,u) + u^\top h(x,u) \right)\\[2mm]
&=&\frac{\partial f^\top}{\partial u}(x,u)p + \frac{\partial h^\top}{\partial u}(x,u)u + h(x,u) 
\end{array}
\eq
Thus an \emph{optimal load} is given by the adjoint variational system \eqref{adjoint} with $u_a=u$
\bq
\label{load}
\begin{array}{rcl}
\dot{p} & = & -\frac{\partial f^\top}{\partial x}(x,u) p - \frac{\partial h^\top}{\partial u}(x,u) u, \quad p(T)=0 \\[2mm]
y_L & = & \frac{\partial f^\top}{\partial u}(x,u) p + \frac{\partial h^\top}{\partial u}(x,u) u,
\end{array}
\eq
with input $u$ equal to the input of the system \eqref{sigma}. In view of the terminal condition $p(T)=0$ this system is \emph{acausal}. Furthermore, in general, it depends both on $p$ \emph{and} on $x$. In the case of a \emph{linear system} \eqref{sigmalin} it reduces to \eqref{adjointlin} with $u_a=u$ and $p(T)=0$, only depending on $p$.
\begin{example}
\end{example}
Let $\Sigma$ be a \emph{static} nonlinearity $y=h(u)$ with $u,y \in \mR^m$. Then an optimal load is given by the static nonlinearity
\bq
\label{optload}
y_L= \frac{\partial h^\top}{\partial u}(u) u,
\eq
while the assumption of invertibility amounts to the mapping
\bq
u \mapsto h(u) + \frac{\partial h^\top}{\partial u}(u) u
\eq
being invertible. This was already derived in \cite{wyattchua}.
%In case the static nonlinearity is generated by a (Rayleigh) potential function $R:\mR^m \to \mR$, i.e., $h$ is given as
%\bq
%h(u) = \frac{\partial R}{\partial u}(u),
%\eq
%then also the optimal load \eqref{optload} admits a potential function. In fact, denoting by $R^*$ the Legendre transform of $R$ and using the involution property $\frac{\partial R^*}{\partial u^*}(\frac{\partial R(u)}{\partial u})=u$ of the Legendre transform, this optimal load equals 
%\bq
%\frac{\partial^2 R}{\partial u \partial u^\top}(u)u = \frac{\partial}{\partial u} \left[R^*(\frac{\partial R}{\partial u}(u))\right]
%\eq
%and thus admits the potential $R^*(\frac{\partial R}{\partial u}(u))$.
%
\begin{example}
\label{ex1}
Consider the linear RC electrical network
\bq
\dot{Q}=I, \quad V=\frac{Q}{C} +RI,
\eq
composed of a capacitor with charge $Q$ and capacitance $C$ in series interconnection with a resistor with resistance $R$. Let the RC network be in series interconnection with a voltage source $V_S$. Then an optimal load is given by the adjoint system
\bq
\dot{p}= - \frac{I}{C}, \; p(T)=0, \quad V_L= p + RI
\eq
Defining $z:=-Cp$ this corresponds to the $RC$ electrical network
\bq
\label{RC}
\dot{z}=I, \; z(T)=0, \quad V_L=\frac{z}{-C} +RI,
\eq
with \emph{negative} capacitance $-C$ and resistance $R$. Note that \eqref{RC} is \emph{not} passive, but only \emph{cyclo-passive} \cite{pass} with \emph{negative} storage function $- \frac{1}{2C}z^2$.
\end{example}

This simple example already illustrates some of the implementation problems of an optimal load, even in the linear case. First, the system \eqref{RC} corresponds to a \emph{negative} capacitance. Second, the system \eqref{RC} is \emph{acausal}, because of the \emph{terminal} condition $z(T)=0$. Of course, these problems are already known from the classical Maximum Power Transfer theorem for linear electrical networks. 

On the other hand, as argued in \cite{desoer}, in principle there may be \emph{other} optimal loads, \emph{different} from \eqref{RC}, with other properties. Indeed, any system that is producing the \emph{same} output voltage $V_L$ as \eqref{RC} when fed by the optimal input current $\hat{I}$ is an optimal load. Furthermore, since $\hat{I}$ depends on the voltage source signal $V_S$, distinction should be made between a load that is optimal for \emph{every} $V_S$, such as \eqref{RC}, or a load that is optimal for one \emph{fixed} $V_S$. This holds for any linear system \eqref{sigmalin}: the adjoint system \eqref{adjointlin} with $u_a=u$ and $p(T)=0$ is an optimal load for \emph{every} $y_S$.

Similar reasoning applies to the nonlinear case. Any system that is producing the \emph{same} $y_L$ as \eqref{load} while fed by the input $u=\hat{u}$ corresponding to $y_S$ is an optimal load for the source signal $y_S$. The adjoint variational system with $u_a=u$ defines an optimal load for \emph{every} $y_S$, where, however, it should be mentioned that the adjoint variational system may depend (through $\hat{x}$) on $y_S$.
\begin{remark}
{\rm
Two directions for extension of the main results are the following. First, instead of imposing the terminal condition $p(T)=0$ on the adjoint variational system one may also compute the optimal input under the constraint of \emph{fixed} terminal condition $x(T)=x_T$. Indeed, this means that variations $\delta u$ should be such that $\delta x(T)=0$. This implies that also in this case $p^\top (T) \delta x (T) = 0$, and thus $\int_0^T y_a^\top (t) \delta u(t) dt = \int_0^T u_a^\top (t) \delta y(t) dt$. Hence \eqref{lambda} and subsequent statements remain unchanged. Furthermore, in this case the assumption of incremental passivity of $\Sigma^+$ may be weakened to incremental \emph{cyclo}-passivity \cite{pass}.

Second extension concerns invertibility of the partial Hessian matrix \eqref{partialHessian}, guaranteeing the existence of the inverse system \eqref{times} of \eqref{+}. If this assumption is not satisfied then one may still exploit the theory of inversion of nonlinear systems by restricting the boundary conditions on the state variables $x,p$. See e.g. \cite{ball} for such an approach in the context of Hamiltonian IO systems.
%Paper with Ball, eq. 48
}
\end{remark}

\section{Special system classes}
\label{special}
In this section we analyze two system classes for which the optimal load \eqref{load} inherits a special form.
\subsection{The linear case}
Example \ref{ex1} generalizes to any linear \emph{port-Hamiltonian system} \cite{pass}
\bq
\label{hamex}
\begin{array}{rcl}
\dot{x} & = & \left( J - R \right) Qx + Bu, \quad J=-J^\top, R=R^\top \geq 0 \\[2mm]
y & = & B^\top Qx +Du, \quad D=D^\top \geq 0,
\end{array}
\eq
where $Q=Q^\top$.
Its adjoint system, defining an optimal load, is given as
\bq
\begin{array}{rcl}
\dot{p} & = & -Q \left(-J - R \right) p - Q Bu_a \\[2mm]
y_a & = & B^\top p +D^\top u_a.
\end{array}
\eq
Define $z:=Q^{-1}p$ (assuming $Q$ invertible) one obtains the equivalent description
\bq
\label{ph}
\begin{array}{rcl}
\dot{z} & = & \left( -J - R \right) (-Q)z + (-B)u_a \\[2mm]
y_a & = & (-B)^\top (-Q)x +D^\top u_a,
\end{array}
\eq
which is again port-Hamiltonian, with $-Q,-J,-B,D^\top$, and same $R$. In particular, if the original port-Hamiltonian system \eqref{hamex} is passive (equivalently $Q\geq 0$), then the optimal load \eqref{ph} is cyclo-passive with nonpositive storage function. 
%Thus the optimal load for a linear port-Hamiltonian system is again port-Hamiltonian, however with Hamiltonian function equal to \emph{minus} the original Hamiltonian function. 
%
%\medskip

Similarly, consider a \emph{linear gradient system} \cite{gradient}
\bq
\begin{array}{rcl}
G\dot{x} & = & -Px + C^\top u \\[2mm]
y & = & Cx +Du,
\end{array}
\eq
with invertible $G=G^\top$ and $P=P^\top, D=D^\top$. Its adjoint system, defining an optimal load, is given as
\bq
\begin{array}{rcl}
\dot{p} & = & PG^{-1}p -C^\top u_a \\[2mm]
y_a & = & CG^{-1} p +D u_a.
\end{array}
\eq
Defining $z:=G^{-1}p$ one obtains the equivalent description
\bq
\begin{array}{rcl}
(-G) \dot{z} & = & -Pz + C^\top u_a \\[2mm]
y_a & = & Cz +D u_a,
\end{array}
\eq
which is again a gradient system, with the same $P$, $C$ and $D$, but with $-G$. 
%Thus, an optimal load is again a gradient system; however with opposite pseudo-inner product.

\subsection{The nonlinear case}
Next, consider a \emph{nonlinear} port-Hamiltonian system
\bq
\begin{array}{rcl}
\dot{x} & = & \left( J - R \right) \frac{\partial H}{\partial x}(x) + Bu \\[2mm]
y & = & B^\top \frac{\partial H}{\partial x}(x) + Du,
\end{array}
\eq
where we assume that $J,R,B$ and $D$ are all \emph{constant} matrices. 
Its adjoint variational systems are given as
\bq
\begin{array}{rcl}
\dot{p} & = & - \frac{\partial^2 H}{\partial x^2}(x) \left(-J - R \right) p -  \frac{\partial^2 H}{\partial x^2}(x) Bu_a \\[2mm]
y_a & = & B^\top p +D^\top u_a.
\end{array}
\eq
Define $z:=\left(\frac{\partial^2 H}{\partial x\partial x^\top}(x)\right)^{-1}p$ to obtain the equivalent form
\bq
\label{ham}
\begin{array}{rcl}
\dot{z} & = & \left( -J - R \right) \left(-\frac{\partial^2 H}{\partial x\partial x^\top}(x)\right)z + (-B)u_a \\[2mm]
y_a & = & (-B)^\top \left(-\frac{\partial^2 H}{\partial x\partial x^\top}(x)\right)z +D^\top u_a,
\end{array}
\eq
which is again port-Hamiltonian, with structure matrices $-J,-B$, and $R$ and Hamiltonian $-\frac{1}{2}z^\top \frac{\partial^2 H}{\partial x\partial x^\top}(x) z$. If the original port-Hamiltonian system is passive (i.e., $H \geq 0$), then the optimal load \eqref{ham} is cyclo-passive with nonpositive Hamiltonian function.

Analogously, consider a \emph{nonlinear} gradient system \cite{gradient}
\bq
\begin{array}{rcl}
G\dot{x} & = & -\frac{\partial V}{\partial x}(x,u) \\[2mm]
y & = & - \frac{\partial V}{\partial u}(x,u),
\end{array}
\eq
where we assume that the pseudo-Riemannian metric matrix $G$ is \emph{constant} (in particular, the metric is flat). Its adjoint system, defining an optimal load, is given as
\bq
\begin{array}{rcl}
\dot{p} & = & -\frac{\partial^2 V}{\partial x\partial x^\top}(x,u)G^{-1}p -  \frac{\partial^2 V}{\partial x \partial u^\top}(x,u)u_a \\[2mm]
y_a & = & -\frac{\partial^2 V}{\partial u \partial x^\top}(x,u)G^{-1} p - \frac{\partial^2 V}{\partial u\partial u^\top}(x,u)u_a
\end{array}
\eq
Defining $z:=G^{-1}p$ one obtains the equivalent description
\bq
\begin{array}{rcl}
(-G) \dot{z} & = & -\frac{\partial^2 V}{\partial x\partial x^\top}(x,u)z - \frac{\partial^2 V}{\partial x \partial u^\top}(x,u) u_a \\[2mm]
y_a & = & -\frac{\partial^2 V}{\partial u \partial x^\top}(x,u)z - \frac{\partial^2 V}{\partial u\partial u^\top}(x,u)u_a,
\end{array}
\eq
defining a gradient system with $-G$ and potential function
\bq
\frac{1}{2} \bma z & u_a \ema^\top 
\bma \frac{\partial^2 V}{\partial x\partial x^\top}(x,u) & \frac{\partial^2 V}{\partial x \partial u^\top}(x,u) \\[2mm]
 \frac{\partial^2 V}{\partial u \partial x^\top}(x,u)     & \frac{\partial^2 V}{\partial u \partial u^\top}(x,u) \ema                              
\bma z \\[2mm] u_a \ema
\eq

\section{Conclusions and outlook}
The Hamiltonian input-output system \eqref{+} turns out to be crucial for the computation of the optimal input and an optimal load. It is of interest to obtain more explicit conditions for its incremental passivity. 

The optimal load provided by the adjoint variational system has been investigated for port-Hamiltonian and gradient systems; however only under the restrictive condition of constant structure matrices and constant pseudo-Riemannian metric.

A major limitation of the results in this paper, as well as in \cite{wyatt}, is the assumption that the system, as seen from the load, can be formulated as an input-state-output system connected to a source. Such an assumed structure is similar to that of a Th\'evenin or Norton equivalent network in electrical network theory. However, the derivation of these equivalent networks involves superposition, and so does not necessarily apply to a general network of nonlinear impedances and sources. Moreover, the assumed splitting of the power variables into inputs and outputs may be restrictive. Recent results in \cite{kron} on Kron reduction of networks of nonlinear resistors could provide a starting point for extensions to general network structures.

\medskip

{\bf Acknowledgement} The work in this paper was stimulated by discussions with John V. Ringwood (Maynooth University, Dublin).

\end{document}